\documentclass{amsart}
\usepackage{amsmath,amssymb,amsthm}
\usepackage{young}
\usepackage{booktabs}
\usepackage[usenames,dvipsnames]{color}
\usepackage[table]{xcolor}
\usepackage{comment}
\usepackage{graphicx}
\usepackage{latexsym}
\usepackage[latin1]{inputenc}
\usepackage{mathtools}
\usepackage[shortlabels]{enumitem}
\usepackage{thmtools}
\usepackage{url}
\usepackage[all,cmtip]{xy}
\usepackage{tikz}
\usepackage[pagebackref]{hyperref} 

\newcommand{\arxiv}[1]{\url{http://arxiv.org/abs/#1}}

\def\a{{\bf a}}
\def\aa{{\bf a}}

\def\ee{{\bf e}}
\def\vv{{\rm v}}

\newcommand{\F}{\mathcal{F}}

\title{Flow polytopes of partitions}

\author{Karola M\'esz\'aros}\address{Department of Mathematics, Cornell University, Ithaca NY}
\email{karola@math.cornell.edu}

\author{Connor Simpson}\address{Department of Mathematics, Cornell University, Ithaca NY}
\email{cgs93@cornell.edu}

\author{Zoe Wellner}\address{Department of Mathematics, Cornell University, Ithaca NY}
\email{zaw5@cornell.edu}

\thanks{M\'esz\'aros is partially supported by a National Science Foundation Grant  (DMS 1501059).}

\definecolor{darkgreen}{rgb}{0,0.7,0}
\definecolor{purplish}{rgb}{0.5,0,0.8}
\definecolor{cobalt}{rgb}{0.0, 0.28, 0.67}
\definecolor{auburn}{rgb}{0.43, 0.21, 0.1}

\hypersetup{
 breaklinks,
colorlinks,citecolor=darkgreen,linkcolor=cobalt,urlcolor=cobalt,
  pdftitle={The polytope of Tesler matrices}
}

\newcommand{\vol}{\operatorname{vol}}
\newcommand{\CT}{\operatorname{CT}}

\newcommand{\Z}{{\mathbb{Z}}}
\def\l{\lambda}
\def\vvv{{\bf v}}

\def\g{{\bf a}}
 \def\f{{\bf f}}
\def\R{\mathbb{R}}
\def\aa{{\bf a}}
\newcommand{\setof}[2]{\left\{ #1 \, : \, #2 \right\}}

\declaretheorem[numberwithin=section]{theorem}
\declaretheorem[numberlike=theorem]{lemma}

\declaretheorem[numberlike=theorem]{corollary}

\declaretheorem[numberlike=theorem, style=definition]{remark}
\declaretheorem[numberlike=theorem, style=definition]{example}

\numberwithin{equation}{section} 

\newtheorem*{theorem1*}{Theorem \ref{thm:vol}}
\newtheorem*{theorem2*}{Theorem \ref{thm:faces}}

\begin{document}

\begin{abstract} Recent progress on flow polytopes indicates many interesting families with product formulas for their volume. These product formulas are all proved using analytic techniques. Our work breaks from this pattern. We define a family of closely related flow polytopes $\F_{(\lambda, {\bf a})}$ for each partition shape $\lambda$ and netflow vector ${\bf a}\in \Z^n_{> 0}$. In each such family, we prove that there is a polytope (the limiting one in a sense) which is a product of scaled simplices, explaining their product volumes. We also show that the combinatorial type  of all polytopes in a fixed family 
 $\F_{(\lambda, {\bf a})}$ is the same. When  $\lambda$ is a staircase shape and ${\bf a}$ is the all ones vector the latter results specializes to a  theorem of the first author with Morales and Rhoades, which shows  that the combinatorial type of the  Tesler polytope is a product of simplices. 
 \end{abstract}

\maketitle

\section{Introduction}\label{sec:intro}

The Catalan numbers, $C_n=\frac{1}{n+1}{n\choose 2}$, $n \in \Z_{\geq 0}$,  are well known for  counting a plethora of combinatorial objects; see \cite[Ex. 6.19]{ec2} for hundreds of interpretations. Naturally then, if an integer polytope has   volume  divisible by a product of consecutive Catalan numbers, one would hope for a combinatorial explanation of such a phenomenon. The latter sentiment ran into obstacles with several flow polytopes, namely the (type A) Chan-Robbins-Yuen polytope \cite{CRY},  its type C and D generalizations \cite{mm}, as well as the Tesler polytope \cite{tesler}. Our work is inspired by the Tesler polytope (which is the flow polytope of the complete graph with netflow vector all ones) $\F_{K_{n+1}}({\bf 1})$, which we explain below can be associated with a staircase partition.  Two of the known main results known about 
$\F_{K_{n+1}}({\bf 1})$ are as follows. We define $\F_{K_{n+1}}({\bf 1})$, and flow polytopes in general, in Section \ref{sec:back}.

\begin{theorem} \label{thm:voltesler} \cite[Theorem 1.9]{tesler}
The normalized volume of the Tesler polytope
$\F_{K_{n+1}}({\bf 1})$  equals
\begin{align}
\vol \F_{K_{n+1}}({\bf 1})
&=  \frac{\binom{n}{2}!\cdot
  2^{\binom{n}{2}}}{\prod_{i=1}^n i!} \notag \\
&= |SYT_{(n-1,n-2,\ldots,1)}| \cdot
\prod_{i=0}^{n-1} C_i, \label{eq:volformula}
\end{align}
where $C_{i}$ is the $i\textsuperscript{th}$ Catalan number
and $|SYT_{(n-1,n-2,\ldots,1)}|$ is the number of Standard Young Tableaux of
staircase shape $(n-1,n-2,\ldots,1)$. 
\end{theorem}

\begin{theorem} 
\label{increasing-simplices} \cite[Corollaries 2.8 \& 2.9]{tesler}
The face poset of the Tesler polytope $\F_{K_{n+1}}({\bf 1})$ is isomorphic to
the face poset of the Cartesian product  of simplices
$\Delta_1 \times \Delta_2 \times \cdots \times \Delta_{n-1}$. In particular, the $h$-polynomial of the Tesler polytope $\F_{K_{n+1}}({\bf 1})$ is the Mahonian distribution
\begin{equation*}
\sum_{i = 0}^{{n \choose 2}} h_i x^i = [n]!_x = (1+x)(1+x+x^2) \cdots (1+x+x^2 + \cdots + x^{n-1}).
\end{equation*}
\end{theorem}

For each partition $\lambda$ and vector ${\bf a}$ we construct a family of flow polytopes $\F_{(\lambda, {\bf a})}$, which we define in Section \ref{sec:family}.  
The Tesler polytope $\F_{K_{n+1}}({\bf 1})$ belongs to $\F_{((n-1,n-2,\ldots,1), {\bf 1})}$. We prove the following general theorems about the families  $\F_{(\lambda, {\bf a})}$.
The limiting polytope $ \F_{(\lambda, {\bf a})}^{\lim}$ is defined in Section \ref{sec:vol}.

\begin{theorem1*}  Let $\lambda$ be a partition, $n \geq \lambda_1 + \ell(\lambda)$, and $\aa \in (\Z_{> 0})^n$ a vector of positive integers. The limiting polytope of $\F_{(\lambda, \aa)}$ is integrally equivalent to a product of scaled simplices $a_1\Delta_{\l_1}\times \cdots \times a_{l(\l)}\Delta_{\l_{l(\l)}}$. Consequently, it has normalize volume
\begin{align}
    \vol \F_{(\lambda, {\bf a})}^{\lim} &=  \left( \sum_{i \in [\ell(\lambda)]}  \lambda_i\right) ! \prod_{i \in [\ell(\lambda)]} \frac{ a_i^{\lambda_i}}{\lambda_i !}
\end{align}\end{theorem1*}

\begin{theorem2*}  
Let  $\lambda=(\l_1, \ldots, \l_k)$ be a partition,  $n$ an integer such that $n-i \geq \lambda_i$ for all $i \in [\ell(\lambda)]$, and $\aa \in \Z_{>0}^n$ a netflow vector.
The face posets of the  polytopes belonging to  $\F_{(\lambda, {\bf a})}$ are isomorphic to the face poset of the Cartesian product  of simplices
$\Delta_{\l_1} \times \Delta_{\l_2} \times \cdots \times \Delta_{\l_k}$. In particular, the $h$-polynomial of   the  polytopes belonging to  $\F_{(\lambda, {\bf a})}$ is  

\begin{equation*}
\sum_{i = 0}^{\sum_{i=1}^k \l_i} h_i x^i = \prod_{i=1}^k [\l_i]_x = \prod_{i=1}^k \left(\sum_{j=0}^{\l_i-1} x^j\right).
\end{equation*}
\end{theorem2*}

In particular, we see that Theorem \ref{increasing-simplices} is a special case of Theorem \ref{thm:faces} for $\F_{K_{n+1}}({\bf 1})$ which belongs to $\F_{((n-1,n-2,\ldots,1), {\bf 1})}$. Also notice the similar volumes for 
$\F_{K_{n+1}}({\bf 1})$ (Theorem \ref{thm:voltesler}) and $\F_{((n-1,n-2,\ldots,1), {\bf 1})}^{\lim}$ (Theorem \ref{thm:vol}); they are off by a factor of $\frac{2^{{n \choose 2}}}{n!}$. We spell this curious fact out in the next corollary.

\begin{corollary}\label{cool}
$\vol \F_{K_{n+1}}({\bf 1})=\frac{2^{{n \choose 2}}}{n!} \vol \F_{((n-1,n-2,\ldots,1), {\bf 1})}^{\lim}$
\end{corollary}

The outline of this paper is as follows. In Section \ref{sec:back} we cover the necessary background and define the class  $\F_{(\lambda, {\bf a})}$. In Section \ref{sec:vol} we define the limiting polytope $\F_{(\lambda, {\bf a})}^{\lim}$ and prove Theorem \ref{thm:vol}. Section \ref{sec:faces} is devoted to proving Theorem \ref{thm:faces}. 
\section{Background and definitions}\label{sec:back}

\subsection{Flow polytopes and Kostant partition functions.} The exposition of this section follows that of \cite{mm}; see \cite{mm} for more details. 

Let $G$ be a (loopless) graph on the vertex set $[n+1]$ with $N$ edges. To each edge $(i, j)$, $i< j$,  of $G$,  associate the positive
type $A_{n}$ root $\vv(i,j)=\ee_i-\ee_j$,
where $\ee_i$ is the $i$th standard basis vector in
$\mathbb{R}^{n+1}$.   Let $S_G := \{\{\vvv_1, \ldots, \vvv_N\}\}$ be the
multiset of roots corresponding to the multiset of edges of $G$. Let  $M_G$ be the
$(n+1)\times N$ matrix whose columns are the vectors in $S_G$.  Fix an
integer  vector $\g=(a_1, \ldots, a_{n+1}) \in \Z^{n+1}$ which we
call the {\bf netflow} and for which we require that $a_{n+1}=-\sum_{i=1}^n a_i$.  An {\bf $\g$-flow} $\f_G$ on $G$ is a 
vector $\f_G=(b_k)_{k \in [N]}$, $b_k \in \R_{\geq 0}$  such that $M_G \f_G=\g$. That is, for all $1\leq i \leq n+1$, we have 

 \begin{equation} \label{eqn:flowA}
\sum_{e=(g<i) \in E(G)} b(e) + a_i = \sum_{e=(i<j) \in E(G)}
b(e)  \end{equation}

Define the {\bf flow polytope} $\F_G(\g)$ associated to a  graph $G$ on the vertex set $[n+1]$ and the integer vector $\aa=(a_1, \ldots, a_{n+1})$ as the set of all $\g$-flows $\f_G$ on $G$, i.e., $\F_G=\{\f_G \in \R^N_{\geq 0} \mid M_G \f_G = \g\}$. The flow polytope 
  $\F_G(\g)$ then naturally lives in $\mathbb{R}^{N}$, where $N$ is the number  of edges of $G$. Note that in order for  $\F_G(\g)$  to be nonempty, it must be that $\sum_{i=1}^{n+1}a_i=0$. For this reason, we also write  $\F_G(a_1, \ldots, a_n):=\F_G(a_1, \ldots, a_n, -\sum_{i=1}^n a_i)$.  The vertices of the flow polytope 
$\F_G(\g)$ are the $\g$-flows whose supports are acyclic subgraphs of $G$ \cite[Lemma 2.1]{hille}.

Recall that the {\bf Kostant partition function}  $K_G$ evaluated at the vector ${\bf b} \in \Z^{n+1}$ is defined as

\begin{equation} \label{kost} K_G({\bf b})= \#  \Big\{ (c_{k})_{k \in [N]}
  \Bigm\vert  \sum_{k \in [N]} c_{k}  \vvv_k ={\bf b} \textrm{ and } c_{k} \in \Z_{\geq 0} \Big\},\end{equation}

\noindent where $[N]=\{1,2,\ldots,N\}$. 

The generating series of the Kostant partition function is
\begin{equation}\label{Kgs}
\sum_{{\bf b} \in \mathbb{Z}^{n+1}} K_G({\bf b}){\bf x}^{\bf b} = {\prod_{(i,j) \in E(G)}} (1-x_ix_j^{-1})^{-1} ,
\end{equation}
where ${\bf x}^{\bf b}=x_1^{b_1}x_2^{b_2}\cdots x_{n+1}^{b_{n+1}}$. In particular, 
\begin{equation} \label{gs:kostant}
K_{K_{n+1}}({\bf b})= [{\bf x}^{\bf b}]
\prod_{1\leq i<j \leq n+1} (1-x_ix_j^{-1})^{-1}.
\end{equation}

Assume that ${\bf a}=(a_1,a_2,\ldots,a_n)$ satisfies $a_i \geq  0$ for
$i=1,\ldots,n$. Let $\aa'=(a_1,a_2,\ldots,a_n, -\sum_{i=1}^n a_i)$. The {\bf generalized Lidskii formulas} of Baldoni and Vergne  state that for a graph $G$ on the vertex set $[n+1]$ with $N$ edges we have
\begin{theorem}\cite[Theorem 38]{BV2}
\begin{equation} \label{eq:vol}
\vol \F_G({\bf a'}) = \sum_{{\bf i}}
\binom{N-n}{i_1,i_2,\ldots,i_n} a_1^{i_1}\cdots
a_n^{i_n}\cdot K_{G'}(i_1-t_1^G, i_2-t_2^G,\ldots, i_n - t_n^G),
\end{equation}
and 
\begin{equation} \label{eq:kost}
K_{G}({\bf a'}) = \sum_{{{\bf i}}}
\binom{a_1+t_1^G}{i_1}\binom{a_2+t_2^G}{i_2} \cdots
\binom{a_{n}+t_n^G}{i_{n}} \cdot  K_{G'}(i_1-t_1^G, i_2-t_2^G,\ldots, i_n - t_n^G),
\end{equation}
where both sums are over weak compositions ${\bf
  i}=(i_1,i_2,\ldots,i_n)$ of $N-n$ with $n$ parts which we
denote as ${\bf i} \models N-n$,  $\ell({\bf i})=n$. The graph $G'$ is the restriction of $G$ to the vertex set $[n]$. The notation  $t_i^G$, $i \in [n]$, stands for the outdegree of vertex $i$ in $G$ minus $1$. 
\end{theorem}

The notation $\vol$ stands for normalized volume. 
  Recall  that the \textbf{Ehrhart polynomial} $i(\mathcal{P}, t)$ of an integer polytope $\mathcal{P} \subset \R^m$ counts the  number of integer points of dilations of  the polytope, $i(\mathcal{P}, t):=\#(t \mathcal{P}\cap \Z^m)$. Its leading coefficient is the {\bf volume} of the polytope.   The  {\bf normalized volume} $\vol(P)$ of a 
$d$-dimensional polytope $\mathcal{P} \subset \mathbb{R}^m$ is the volume form which 
 assigns a volume of one to the smallest $d$-dimensional integer simplex in the affine span of $\mathcal{P}$. In other words, the normalized volume of a
$d$-dimensional polytope $\mathcal{P}$ is $d!$ times its  volume.

\subsection{The family $\F_{(\lambda, {\bf a})}$} \label{sec:family}
We start by  defining  a family of graphs associated to the partition $\lambda$.
Given a partition $\lambda$, let $Y$ be the left-justified Young diagram corresponding to $\lambda$. Pick an integer $n$ such that $n-i \geq \lambda_i$ for all $i \in [\ell(\lambda)]$. We can place $Y$ inside the upper triangle (not including the diagonal) of an $n \times n$ matrix $M$, with the top and right edges of $Y$ flush with the top and right edges of $M$. Now, let $Y'$ be the set of entries $(i,j)$ of $M$ that lie inside $Y$, and define $G(\lambda, n)$ to be the directed graph 
\[ G(\lambda, n):=\Big( [n+1], \{ (i,n+1) \, : \, i \in [n]\} \cup Y' \Big). \]
\begin{example}
    Construction of $G( (2, 1, 1), 5))$. 
    \begin{figure}[h]
	\centering
	\begin{minipage}{0.25in}
	    \reflectbox{
	    \begin{Young}
		& \cr
		\cr
		\cr
	    \end{Young}
	}
	\end{minipage}
	\hspace{0.5in}
       \begin{minipage}{1.75in}
	   \newcommand{\colorme}{\cellcolor{gray!30}}
	   $\begin{array}{|c|c|c|c|c|r}
	       \multicolumn{1}{c}{1} & \multicolumn{1}{c}{2} & \multicolumn{1}{c}{3} & \multicolumn{1}{c}{4} & \multicolumn{1}{c}{5} \\
	       \cline{1-5}
	       + &  &  & \colorme * & \colorme * & 1 \\ \cline{1-5}
	       & + &  &  & \colorme 	* & 2 \\ \cline{1-5}
	       &  & + &  & \colorme 	* & 3 \\ \cline{1-5}
	       &  &  & + &  		& 4 \\ \cline{1-5}
	       &  &  &  & +  		& 5 \\ \cline{1-5}
	   \end{array}$
       \end{minipage}
    \begin{minipage}{2in}
	\begin{tikzpicture}
	    \foreach \x in {1,...,6} {
		\node (\x) at (\x, 0) [draw,fill,circle,scale=0.3] {};
		\node (\x label) at (\x, -0.3) {\x} ;
	    }
	    \draw (1) to [bend left=60] (4) ;
	    \draw (1) to [bend left=60] (5) ;
	    \draw (2) to [bend left=60] (5) ;
	    \draw (3) to [bend left=60] (5) ;
	    \foreach \x in {1,...,6} {
		\draw (\x) to [bend left=60] (6) ;
	    }
	\end{tikzpicture}
    \end{minipage}
    \caption{From left to right: the left-justified Young diagram of $\lambda = (2,1,1)$, the diagram in a $5 \times 5$ matrix, and the corresponding graph on six vertices.}
    \end{figure}
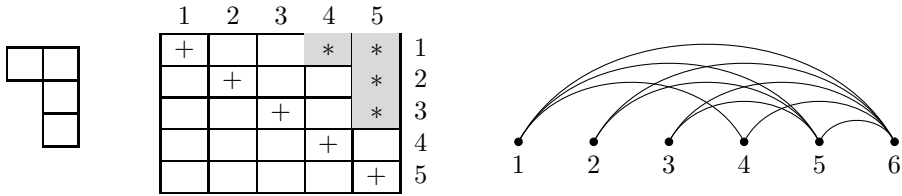
\end{example}

For a vector $\aa \in \Z_{>0}^m$ with $m \geq \lambda_1 + \ell(\lambda)$, define the family 
\[ \F_{(\lambda, \aa)} := \setof{ \F_{G(\lambda, n)}(\aa) }{ \max(\lambda_1, \ell(\lambda)) < n \in \Z}. \] 
Note that there is a small abuse of notation in the definition above: if $n \neq m$, then $\aa$ will have too many or too few entries to serve as a netflow for many $G(\lambda, n)$. When $n \leq m$, then we can just use the first $n$ entries of $\aa$. For $n > m$, we show in Section \ref{sec:vol} that the choice of additional entries is irrelevant: any element of $\Z_{> 0}^n$  whose first $m$ entries match those of $\aa$ will product essentially the same polytope. More precisely, we prove that all the above mentioned polytopes are integrally equivalent. Recall  that integer polytopes $\mathcal{P}\subset \R^m$ and $\mathcal{Q}\subset \R^k$ are \textbf{integrally equivalent} if there is an affine transformation $f:\R^m\rightarrow \R^k$ such that $f$ maps $\mathcal{P}$ bijectively onto $\mathcal{Q}$ and $f$ maps $\Z^m\cap  \operatorname{aff}(\mathcal{P})$ bijectively onto $\Z^k\cap \operatorname{aff}(\mathcal{Q})$, where  $ \operatorname{aff}$ denotes affine span. If two polytopes are  integrally equivalent, then they have the same combinatorial type as well as the same volume and more generally the same Ehrhart polynomial.

Observe that for any $n \in \Z_{>0}$ and $\lambda = (n-1, n-2, \ldots, 1)$, $G(\lambda, n) = K_{n+1}$. Setting $\aa = {\bf 1}$, it follows that  the Tesler polytope $\F_{K_{n+1}}({\bf 1})$ belongs to $\F_{((n-1, n-2, \ldots, 1), {\bf 1})}$.
\section{The limiting polytopes of $\F_{(\lambda, {\bf a})}$}\label{sec:vol}

In this section we define the limiting polytope of the family $\F_{(\lambda, {\bf a})}$ for any partition $\lambda$ and netflow vector $\a$. We then establish the combinatorial structure and the volume of these limiting polytopes. 

One can easily see the need to define a limiting polytope of $\F_{(\lambda, {\bf a})}$ from the following data on the normalized volumes of the members of the family $\F_{((4,3,2,1), {\bf 1})}$:
\begin{center}
    \begin{tabular}{r|lllllll}
	$n $ & 5 & 6 & 7 & 8 & 9 & 10 & 11  \\
	$\vol \F_{G( (4,3,2,1),n)}({\bf 1}) $ & 107520 & 26580 & 15120 & 12600 & 12600 & 12600 & 12600 
    \end{tabular}
\end{center}
One immediately notices that the volume of the polytopes in question appears to stabilize for large $n$. This is not a coincidence, and is in fact a general feature of polytopes in $\F_{(\lambda, \aa)}$, as we show in this section.

For a partition $\lambda$ and $\aa \in \Z_{> 0}^n$, define the \textbf{limiting polytope of the family} $\F_{(\lambda, \aa)}$, denoted $\F_{(\lambda, \aa)}^{\lim}$, to be the polytope $\F_{G(\lambda, \ell(\lambda) + \lambda_1)}(\aa)$. We prove  in Lemma \ref{lem:product} 
that for all $n \geq \ell(\lambda) + \lambda_1$ we have that $\F_{G(\lambda,n)}(\aa)$ and $\F_{G(\lambda, \ell(\lambda) + \lambda_1)}(\aa)$ are integrally equivalent; thus any one of $\F_{G(\lambda,n)}(\aa)$ with $n \geq \ell(\lambda) + \lambda_1$ can be thought of as  $\F_{(\lambda, \aa)}^{\lim}$. 

\subsection{Structure and Volume of the Limiting Polytope}

Given a graph $G = G(\lambda, n)$, 
for each vertex $i \in [n]$, let
    \[ G_i = ( [n+1], \setof{(i,j) \in E(G)}{i < j} \cup \setof{ (j,n+1)}{j  \in [n]}) \]
    be the subgraph of $G$ graph obtained by restricting $E(G)$ to those edges that come out of vertex $i$ or go to the sink.

\begin{lemma}
    Let $\lambda$ be a partition, let $n \geq \ell(\lambda) + \lambda_1$, let $G$ and $G_i$ be as above for $i \in [n]$, and let $\aa \in \Z_{> 0}^n$. 
    Then $\F_{G(\lambda, n)}(\aa)$ is integrally equivalent to $\prod_{i = 1}^n \F_{G_i}(\aa)$:
    $$\F_{G(\lambda, n)}(\aa) \equiv  \prod_{i = 1}^n \F_{G_i}(\aa).$$
    \label{lem:product}
\end{lemma}
\begin{proof}
    \newcommand{\ph}{\varphi}
    Define the map  $\ph: \F_G(\aa) \to \prod_{i = 1}^n \F_{G_i}(\aa)$ by 
    \[ \ph(f) = (f_1, \ldots, f_n) \]
    where $f_i : E(G_i) \to \R$ is defined by
    \[ f_i(p,q) = \begin{cases}
    	f(i,j), & (p,q) = (i,j) \\
    	a_p + f(i,p), &  q = n+1 \textrm{ and } (i,p) \in E(G_i) \\
    	a_p, & q = n+1 \textrm{ and } (i,p) \not\in E(G_i)
    \end{cases} \]

    The inverse of map $\ph$ is $\ph^{-1}:\prod_{i = 1}^n \F_{G_i} \to \F_G$ defined by $\ph^{-1}(f_1, \ldots, f_n) = f$ where 
 $ f(p,q) = f_p(p,q),$
     thus $\ph$ is a bijection between $\F_{G(\lambda, n)}(\aa)$ and $\prod_{i = 1}^n \F_{G_i}(\aa)$. Moreover, $\ph$ can be extended to an affine map mapping the integer points of the affine span of $\F_{G(\lambda, n)}(\aa)$ bijectively to the integer points of the affine span of $\prod_{i = 1}^n \F_{G_i}(\aa)$, concluding the proof.\end{proof}

 We now show that the polytopes $\F_{G_i}(\aa)$ appearing in Lemma \ref{lem:product} are very special:
    \begin{lemma}
	For $i \in [\ell(\lambda)]$, $\aa \in \Z_{> 0}^n$, $\F_{G_i}(\a)$ is integrally equivalent  to $a_i \Delta_{\lambda_i}$, a scaled simplex of dimension $\lambda_i$. For $\ell(\lambda) < i < n+1$, $\F_{G_i}(\a)$ is a point.
	\label{lem:simplices}
    \end{lemma}
    \begin{proof}
	\newcommand{\ph}{\varphi}
	\renewcommand{\v}{\mathbf{v}}
	Let $i \in [\ell(\lambda)]$.
	Define $\ph_i : \F_{G_i}(\aa) \to a_i \Delta_{\lambda_i}$ by 
	$\ph_i(f_i) = \v \in  \R^{\lambda_i+1}$, where $v_j = f_i(i, n +2- j)$. To see that this function is well-defined, note that $i \in V(G_i)$ has no incoming edges (see Figure \ref{fig:bign}), so $\sum_{(i,j) \in E(G_i)} f(i,j)  = \sum_{j \in [\lambda_i+1]} v_j = a_i$.
	This map is a projection;  it is affine and preserves integer points. It is not hard to see that $\ph_i $ is a bijection between  $\F_{G_i}(\aa)$ and $a_i \Delta_{\lambda_i}$.
	Furthermore, the second claim of Lemma \ref{lem:simplices} is immediate. \end{proof}

   Lemmas \ref{lem:product} and \ref{lem:simplices} imply that we can consider any one of  $\F_{G(\lambda,n)}(\aa)$ with $n \geq \ell(\lambda) + \lambda_1$  as the limiting polytope $\F_{(\lambda, \aa)}^{\lim}$. Indeed, when $n \geq \lambda_1 + \ell(\lambda)$, it is guaranteed that the Young diagram of $\lambda$ will fit in the top right quadrant of an $n \times n$ matrix. Figure  \ref{fig:bign}  illustrates the effects of this. 
   
     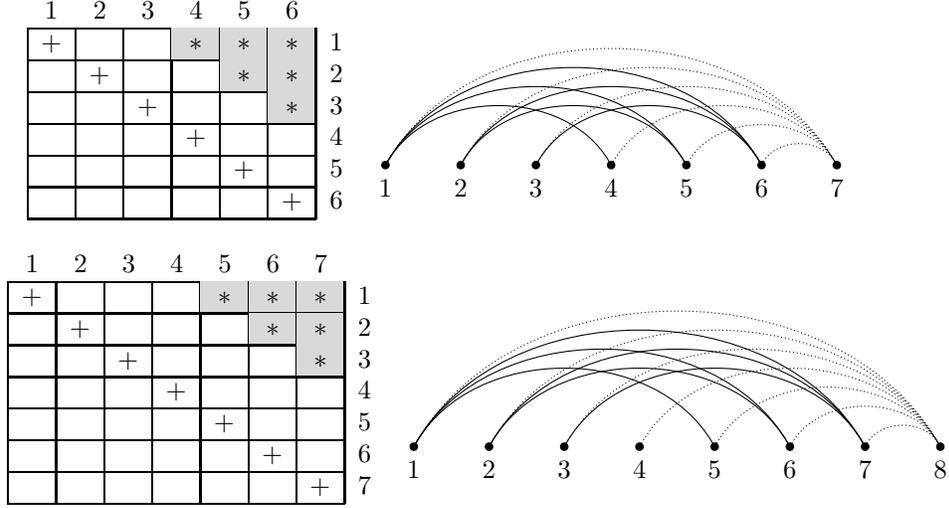
\begin{figure}[h] 
       \begin{minipage}{1.75in}
	   \newcommand{\colorme}{\cellcolor{gray!30}}
	   $\begin{array}{|c|c|c|c|c|c|r}
	       \multicolumn{1}{c}{1} & \multicolumn{1}{c}{2} & \multicolumn{1}{c}{3} & \multicolumn{1}{c}{4} & \multicolumn{1}{c}{5} & \multicolumn{1}{c}{6}\\
	       \cline{1-6}
	       +& &  & \colorme * & \colorme * & \colorme * & 1 \\ \cline{1-6}
	       &+ &  &  & \colorme * & \colorme * & 2 \\ \cline{1-6}
	       && + &  &  & \colorme 	* & 3 \\ \cline{1-6}
	       &&  & + &  &  & 4 \\ \cline{1-6}
	       &&  &  & + &  		& 5 \\ \cline{1-6}
	       &&  &  &  & +  		& 6 \\ \cline{1-6}
	   \end{array}$
       \end{minipage}
    \begin{minipage}{2in}
	\begin{tikzpicture}
	    \foreach \x in {1,...,7} {
		\node (\x) at (\x, 0) [draw,fill,circle,scale=0.3] {};
		\node (\x label) at (\x, -0.3) {\x} ;
	    }
	    \draw (1) to [bend left=60] (4) ;
	    \draw (1) to [bend left=60] (5) ;
	    \draw (1) to [bend left=60] (6) ;
	    \draw (2) to [bend left=60] (5) ;
	    \draw (2) to [bend left=60] (6) ;
	    \draw (3) to [bend left=60] (6) ;
	    \foreach \x in {1,...,7} {
		\draw[densely dotted] (\x) to [bend left=60] (7) ;
	    }
	\end{tikzpicture}
    \end{minipage} \vspace{0.15in}\\\noindent
       \begin{minipage}{2.0in}
	   \newcommand{\colorme}{\cellcolor{gray!30}}
	   $\begin{array}{|c|c|c|c|c|c|c|r}
	       \multicolumn{1}{c}{1} & \multicolumn{1}{c}{2} & \multicolumn{1}{c}{3} & \multicolumn{1}{c}{4} & \multicolumn{1}{c}{5} & \multicolumn{1}{c}{6}& \multicolumn{1}{c}{7}\\
	       \cline{1-7}
	       +&&&& \colorme *& \colorme * & \colorme * & 1 \\ \cline{1-7}
	       &+&&&& \colorme *& \colorme 	* & 2 \\ \cline{1-7}
	       &&+&&&& \colorme * & 3 \\ \cline{1-7}
	       &&&+&&&  		& 4 \\ \cline{1-7}
	       &&&&+&&  		& 5 \\ \cline{1-7}
	       &&&&&+&  		& 6 \\ \cline{1-7}
	       &&&&&&+  		& 7 \\ \cline{1-7}
	   \end{array}$
       \end{minipage}
    \begin{minipage}{2in}
	\begin{tikzpicture}
	    \foreach \x in {1,...,8} {
		\node (\x) at (\x, 0) [draw,fill,circle,scale=0.3] {};
		\node (\x label) at (\x, -0.3) {\x} ;
	    }
	    \draw (1) to [bend left=60] (5) ;
	    \draw (1) to [bend left=60] (6) ;
	    \draw (1) to [bend left=60] (7) ;
	    \draw (2) to [bend left=60] (6) ;
	    \draw (2) to [bend left=60] (7) ;
	    \draw (3) to [bend left=60] (7) ;
	    \foreach \x in {1,...,8} {
		\draw[densely dotted] (\x) to [bend left=60] (8) ;
	    }
	\end{tikzpicture}
    \end{minipage}
    \caption{The Young diagram of $\lambda = (3,2,1)$ in both $6 \times 6$ and $7 \times 7$ matrices, and the corresponding graphs $G(\lambda, 6)$ and $G(\lambda, 7)$, with edges to the sink dotted. Observe that increasing $n$ by 1 adds a single new vertex with a single outgoing edge to the sink. This underlies the fact that $\F_{G(\lambda, 6)}(\aa)$ and $\F_{G(\lambda, 7)}(\aa)$ are integrally equivalent and have the same volume. It also justifies our use of $\aa$ as the netflow vector for both $\F_{G(\lambda, 6)}$ and $\F_{G(\lambda, 7)}$: only the first $\ell(\lambda)$ entries of the netflow vector matter.}
    \label{fig:bign}
    \end{figure}

    The decomposition of the limiting polytope into simplices also gives us a neat formula for its volume.
    
\begin{theorem} \label{thm:vol} 
    Let $\lambda$ be a partition, $n \geq \lambda_1 + \ell(\lambda)$, and $\aa \in (\Z_{> 0})^n$ a vector of positive integers. The limiting polytope of $\F_{(\lambda, \aa)}$ is integrally equivalent to a product of scaled simplices $a_1\Delta_{\l_1}\times \cdots \times a_{l(\l)}\Delta_{\l_{l(\l)}}$. Consequently, it has normalize volume
\begin{align}
    \vol \F_{(\lambda, {\bf a})}^{\lim} &=  \left( \sum_{i \in [\ell(\lambda)]}  \lambda_i\right) ! \prod_{i \in [\ell(\lambda)]} \frac{ a_i^{\lambda_i}}{\lambda_i !}
\end{align}
\end{theorem}
\begin{proof}
    It is immediate from Lemmas \ref{lem:product} and \ref{lem:simplices} that 
    $\F_{(\lambda, \aa)}^{\lim}$ is integrally equivalent to $\prod_{i \in [\ell(\lambda)]} a_i \Delta_{\lambda_i}$.
    The unnormalized volume of $a_i\Delta_{\lambda_i}$ is $\frac{a_i^{\lambda_i}}{\lambda_i!}$, so the unnormalized volume of $\F_{(\lambda, \aa)}^{\lim}$ is $\prod_{i \in \ell(\lambda)} \frac{a_i^{\lambda_i}}{\lambda_i !}$. To normalize this volume, we divide it by the volume of the standard simplex of dimension $\dim \F_{(\lambda, \aa)}^{\lim}$. $\F_{(\lambda, \aa)}^{\lim}$ has dimension $\sum \lambda_i$, so our ending expression is 
    \[ \vol \F_{(\lambda, \aa)}^{\lim} = \left(\sum_{i \in \ell(\lambda)} \lambda_i\right)!\prod_{i \in \ell(\lambda)} \frac{a_i^{\lambda_i}}{\lambda_i!} \]
    where $\vol \F_{(\lambda, \aa)}^{\lim}$ denotes the normalized volume of $\F_{(\lambda, \aa)}^{\lim}$.
\end{proof}

We note that we can relax the requirement $\aa \in (\Z_{> 0})^n$ to $\aa \in (\Z_{\geq 0})^n$ and obtain similar results. Indeed, both  Lemmas \ref{lem:product}  and  \ref{lem:simplices} and their proofs hold verbatim (if $a_i=0$ then $\F_{G_i}(\aa)$ is a point). Thus, the analogues of Lemmas 
\ref{lem:product}  and \ref{lem:simplices} yield a volume formula for any $ \F_{(\lambda, {\bf a})}^{\lim}$,  $\aa \in (\Z_{\geq 0})^n$. For simplicity, we will work with $\aa \in (\Z_{> 0})^n$ throughout the paper.

  \subsection{Constant Term Identities}
{   
 
    \renewcommand{\i}{\mathbf{i}}
    \newcommand{\x}{\mathbf{x}}
    Using the volume formula given in Theorem \ref{thm:vol}, we can derive a   constant term identity.
Let $\lambda$ be a partition, $n$ an integer such that $n-i \geq \lambda_i$ for all $i \in [\ell(\lambda)]$, and $\aa \in \Z_{>0}^n$.
For convenience, let $L = \sum_{i \in [\ell(\lambda)]} \lambda_i$ and let $G$ be the restriction of $G(\lambda, n)$ to the vertex set $[n]$.
Further, let $\bar\lambda = (\lambda_1, \ldots, \lambda_{\ell(\lambda)}, 0, \ldots, 0,0) \in \Z^n$.
 
 \begin{theorem}
    \label{thm:ct}
    $$CT_{x_n} \ldots \CT_{x_1}
	(a_1 x_1 + \cdots + a_n x_n)^L
	\prod_{i \in [\ell(\lambda)]} \prod_{n + 1 - \lambda_i \leq j \leq n} (x_i - x_j)^{-1}]= \left( \sum_{i \in [\ell(\lambda)]}  \lambda_i\right) ! \prod_{i \in [\ell(\lambda)]} \frac{ a_i^{\lambda_i}}{\lambda_i !}$$

    \end{theorem}
\begin{proof}
    By Equation \eqref{eq:vol}, the volume of $\F_{(\lambda, \aa)}^{\lim}$ is equal to 
    \begin{equation*}
	\vol \F_{(\lambda, \aa)}^{\lim} 
	= 
	\sum_{\substack{\i \vDash L \\ \ell(\i) = n}} 
	\binom{L}{i_1, \ldots, i_n}
	\left( \prod_{j \in [n]} a_j^{i_j} \right) 
	K_G(\i - \bar\lambda).
    \end{equation*}
    Now, let $G'$ be $G$ with all its edges reversed and observe that $K_G(\i - \bar\lambda) = K_{G'}(\bar\lambda - \i)$. Thus, the above is equal to
    \begin{align*}
	=& 
	\sum_{\substack{\i \vDash L \\ \ell(\i) = n}} 
	\binom{L}{i_1, \ldots, i_n}
	\left( \prod_{j \in [n]} a_j^{i_j} \right) 
	K_{G'}(\bar\lambda - \i) \\
	=& 
	\sum_{\substack{\i \vDash L \\ \ell(\i) = n}} 
	\binom{L}{i_1, \ldots, i_n}
	\left( \prod_{j \in [n]} a_j^{i_j} \right) 
	[\x^{\bar\lambda - \i}] \prod_{(i,j) \in E(G)} (1 - x_j x_i^{-1})^{-1} \\
	=& 
	\CT_{x_n} \ldots \CT_{x_1}
	\sum_{\substack{\i \vDash L \\ \ell(\i) = n}} 
	\binom{L}{i_1, \ldots, i_n}
	\left( \prod_{j \in [n]} a_j^{i_j} \right) 
	\x^{\i - \bar\lambda} \prod_{(i,j) \in E(G)} (1 - x_j x_i^{-1})^{-1}.
    \end{align*}
    Since the $i$th vertex of $G$ has $\bar\lambda_i$ edges out of it, 
    $\prod_{(i,j) \in E(G)} (1 - x_j x_i^{-1})^{-1} = \x^{\bar\lambda} \prod_{(i,j) \in E(G)} (x_i - x_j)^{-1}$. It follows that the above is equal to
    \begin{align*}
	=& \CT_{x_n} \ldots \CT_{x_1}
	\sum_{\substack{\i \vDash L \\ \ell(\i) = n}} 
	\binom{L}{i_1, \ldots, i_n}
	\left( \prod_{j \in [n]} a_j^{i_j} \right) 
	\x^{\i} \prod_{(i,j) \in E(G)} (x_i - x_j )^{-1} \\
	=& \CT_{x_n} \ldots \CT_{x_1}
	(a_1 x_1 + \cdots + a_n x_n)^L
	\prod_{(i,j) \in E(G)} (x_i - x_j )^{-1}
    \end{align*}
    where the latter equality follows by the Multinomial Theorem.
   The   product in the last expression can be rewritten as  $\prod_{i \in [\ell(\lambda)]} \prod_{n + 1 - \lambda_i \leq j \leq n} (x_i - x_j)^{-1}$. 
    Finally, substituting in the formula for $\vol \F_{(\lambda, \aa)}^{\lim}$ given in Theorem \ref{thm:vol} yields the result.
\end{proof}
}
 
\section{The face structure of  polytopes in $\F_{(\lambda, {\bf a})}$}\label{sec:faces}
 
In Theorem \ref{thm:vol}, we showed that for all $\lambda$ and $\aa \in \mathbb{Z}^n_{>0}$, $\F_{(\lambda,\aa)}^{\lim}$ is integrally equivalent to a product of simplices, implying that its combinatorial type is that of a product of simplices.
In this section, we show that each element of the family $\F_{(\lambda, \aa)}^{\lim}$ has the same combinatorial type as $\F_{(\lambda,\aa)}^{\lim}$.

\subsection{A quick review of results relating subgraphs and the face lattice}
Before proceeding, we will review some facts relating the face lattice of a flow polytope to subgraphs of the graph from which it arrises.

Let $G$ be a graph and $\aa$ a netflow vector. We call a subgraph $H$ of $G$ \textbf{$\aa$-regular} (or just \textbf{regular} when the netflow in question is clear) if there is an $\aa$-flow $f$ on $G$ such that $f$ is zero on all edges of $G$ that are not in $H$. 
We say that $\aa$ is in \textbf{generic position} with respect to $G$ if there is no $\aa$-flow $f$ such that $f$ is the unique flow on two distinct subtrees of $G$.
The following two results are implied by \cite[Lemma 2.1 \& Theorem 2.2]{hille} for the faces of $\F_{G_{(\lambda, n)}}(\aa)$.  
 
\begin{lemma} 
    \label{lem:hille-verts}
    The vertices of $\F_{G_{(\lambda, n)}}(\aa)$ are the flows on the regular subtrees of ${G_{(\lambda, n)}}$.
\end{lemma}
\begin{theorem} 
    \label{thm:hille-faces}
     If $\aa$ is in generic position, then the regular subtrees of ${G_{(\lambda, n)}}$ are in bijection with the vertices of $\F_{G_{(\lambda, n)}}(\aa)$ and the faces of $\F_{G_{(\lambda, n)}}(\aa)$ are in bijection with the regular subgraphs of ${G_{(\lambda, n)}}$.
\end{theorem}
\subsection{Characterization of regular subtrees}
    Let $\lambda = (\lambda_1, \ldots, \lambda_\ell)$ be a partition, $n$ an integer such that $n-i \geq \lambda_i$ for all $i \in [\ell(\lambda)]$, $G = G(\lambda, n)$, and let $\a \in \Z_{>0}^n$.
    In this section, we  characterize which subtrees of $G$ are $\a$-regular for $\a \in \Z_{>0}^n$.
\begin{lemma} \label{lem:tree}
    Let $H$ be a subgraph of $G$ built by picking one outgoing edge from each vertex $i < n+1$. Then $H$ is an $\a$-regular spanning tree of $G$ for $\a \in \Z_{>0}^n$.
\end{lemma}
\begin{proof}
    First we  show that $H$ is acyclic and connected. If there were a cycle $C \subset H$, then there would have to be two outgoing edges from its minimal vertex. Thus, $H$ is acyclic.  To see that $H$ is connected, we note that our graph has $n$ edges, $n+1$ vertices and it has no cycles. 

    To see that $H$ is regular, construct an $\a$-flow on it as follows. Let $e_v \in E(H)$ be the unique edge out of $v$ in $H$. Let $f(e_1) = a_1$. Assume $f(e_1), \ldots, f(e_{i})$ have been assigned for some $i \geq 1$. Then we let   $f(e_{i+1}) = a_{i+1}$ if $i+1$ has no incoming edges in $H$, and  $f(e_{i+1}) = a_{i+1} + \sum_{(v,i+1) \in E(H)} f(e_v)$ for vertices with incoming edges.
\end{proof}
\begin{lemma} \label{lem:oneedge}
    Every spanning subtree $T$ of $G$ that admits an $\a$-flow for $\a \in \Z_{>0}^n$ has a single edge out of each of its vertices $v < n+1$.
\end{lemma}
\begin{proof}
Suppose that a spanning subtree $T$ of $G$ has at least two outgoing edges from a vertex $v$. Since $T$ has $n$ edges and $n+1$ vertices,  it follows then that $T$ has two vertices with no outgoing edges. In particular, there is a $v<n+1$ with only incoming edges. Since $a_v>0$, such a tree cannot admit an $\a$-flow. Thus each spanning subtree $T$ of $G$ that admits an $\a$-flow has at most one edge out of each of its vertices $v < n+1$. Since we need $n$ edges, it has exactly one edge out of each of its vertices $v < n+1$.
\end{proof}
\begin{theorem}\label{thm:vertexcount}
    The $\a$-regular subtrees for $\a \in \Z_{>0}^n$ of $G(\lambda, n)$ are precisely those that have exactly one edge out of every vertex.
    The polytope $\F_{G(\lambda, n)}(\aa)$ has $\prod_{i = 1}^{\ell(\lambda)} (\lambda_i + 1)$ vertices, independent of $n$, corresponding to the $\a$-flows on the aforementioned subtrees.
\end{theorem}
\begin{proof}
    The first statement follows immediately from  Lemmas \ref{lem:tree} and \ref{lem:oneedge}. 
    We can count such trees by noting that each vertex $i$ has $1 + \lambda_i$ edges out of it (and 1 edge if $i > \ell$), so there are $\prod_{i = 1}^\ell (\lambda_i + 1)$ ways to choose such a tree.

    Now, note that if $T$ is a spanning tree of $G$ that admits a flow $f_T$, then $f_T$ must be nonzero on every edge of $T$ because $\a$'s non-sink entries are all positive. Thus, $\a$ is in generic position and by Lemma \ref{thm:hille-faces}
 there is a bijection between subtrees of $G$ that admit regular flows and vertices of $\F_G(\a)$.
\end{proof}
\subsection{The face lattice}
We are ready  to show that every polytope in $\F_{(\lambda, \aa)}$ has a face lattice isomorphic to that of $\F_{(\lambda, \aa)}^{\lim}$.

\begin{lemma}
    \label{lem:regular}
    Let $\lambda$ be a partition, $n$ an integer such that $n-i \geq \lambda_i$ for all $i \in [\ell(\lambda)]$, and $\aa \in \Z_{>0}^n$.
    The regular subgraphs of $G(\lambda, n)$ are precisely those that have at least one edge out of every non-sink vertex.
Furthermore, for $H$ and $K$ regular subgraphs of $G(\lambda, n)$, $\F_K(\aa) \subset \F_H(\aa)$ if and only if $K$ is a subgraph of $H$. Thus,  the face lattice of $\F_{G(\lambda, n)}(\aa)$ is isomorphic to the poset of regular subgraphs of $G(\lambda, n)$.
\end{lemma}
\begin{proof}
    The entries of $\aa$ are all positive, so every vertex of a regular subgraph must have at least one outgoing edge. Conversely, any subgraph that has at least one edge out of every non-sink vertex contains a regular subtree by Theorem \ref{thm:vertexcount} and is therefore regular.

    For the second statement, the ``if'' implication is clear. For the ``only if'', observe that if $e$ is an edge in $K$ that is not in $H$, then there is a regular subtree $T$ contained in $K$ such that $e \in E(T)$.  Since $\aa$ is in generic position, the unique flow $f$ on $T$ is nonzero on $e$ and is therefore not in $\F_H(\aa)$, so $\F_K(\aa) \not\subset \F_H(\aa)$. The last statement then follows by Lemma \ref{thm:hille-faces} since $\a$ is in generic position.
\end{proof}
\begin{theorem} \label{thm:faces} 
    Let  $\lambda=(\l_1, \ldots, \l_k)$ be a partition,  $n$ an integer such that $n-i \geq \lambda_i$ for all $i \in [\ell(\lambda)]$, and $\aa \in \Z_{>0}^n$ a netflow vector.
The face posets of the  polytopes belonging to  $\F_{(\lambda, {\bf a})}$ are isomorphic to the face poset of the Cartesian product  of simplices
$\Delta_{\l_1} \times \Delta_{\l_2} \times \cdots \times \Delta_{\l_k}$. In particular, the $h$-polynomial of   the  polytopes belonging to  $\F_{(\lambda, {\bf a})}$ is  

\begin{equation*}
\sum_{i = 0}^{\sum_{i=1}^k \l_i} h_i x^i = \prod_{i=1}^k [\l_i]_x = \prod_{i=1}^k \left(\sum_{j=0}^{\l_i-1} x^j\right).
\end{equation*}
\end{theorem}
\begin{proof}
    \newcommand{\ph}{\varphi}
    \newcommand{\C}{\mathcal{C}}
    \newcommand{\Reg}{\mathcal{R}}
    Let $\bar\lambda = (\lambda_1 + 1, \lambda_2 + 1, \ldots, \lambda_k + 1)$ and 
    let $Y$ be the Young diagram of $\bar\lambda$. Let $\C$ be the poset of subsets $C$ of the boxes of $Y$ such that $C$ contains at least one box from every row of $Y$, ordered by inclusion.

    For any integer $n$ such that $n-i \geq \lambda_i$ for all $i \in [\ell(\lambda)]$,  define the  bijection $\ph_n: Y \to E(G(\lambda, n))$ by $\ph_{n}(i,j) =(i,n+2-j)$. 
    Now, let $\Reg(G(\lambda, n))$ be the set of regular subgraphs of $G$ and define $\phi_n : \C \to \Reg(G(\lambda, n))$ by 
    \[ C \mapsto \Big( [n+1], \ph_n(C) \cup \setof{ (i,n+1)}{n+1 > i > \ell(\lambda)} \Big). \]
    Every subgraph in the image of $\phi_n$ has an edge out of every vertex besides $n+1$, so by Lemma \ref{lem:regular}, the image of $\phi_n$ lies in $\Reg(G(\lambda, n))$ as claimed. In fact, using the surjectivity of $\ph_n$,  the image of $\phi_n$ is all of $\Reg(G(\lambda, n))$.
    Finally, injectivity of $\phi_n$ follows from the fact that $\ph_n$ is injective, and it is clear that it preserves inclusion. Therefore, $\phi_n$ is an order preserving bijection between $\C$ and $\Reg(G(\lambda, n))$.
    Applying the second statement of Lemma \ref{lem:regular}, we have that for any   integer $n$ such that $n-i \geq \lambda_i$ for all $i \in [\ell(\lambda)]$, the face lattice of $\F_{G(\lambda, n)}(\aa)$ is isomorphic to $\C$.
In particular, the  face lattice of $\F_{(\lambda,\aa)}^{\lim}=\F_{G(\lambda, \lambda_1 + \ell(\lambda))}(\aa)$ is isomorphic to $\C$. The former is isomorphic to the face lattice of $\Delta_{\lambda_1} \times \Delta_{\lambda_2} \times \cdots \times \Delta_{\lambda_k}$ by Theorem \ref{thm:vol}.   It follows that every element of $\F_{(\lambda,\aa)}$ has a face lattice isomorphic to this one. Furthermore,  $\F_{(\lambda,\aa)}$ is of dimension $\sum_{i=1}^k \l_i$ and its $h$-polynomial is given by the products of the $h$-polynomials of the simplices: $h(\F_{(\lambda,\aa)}, x)=\prod_{i=1}^k h(\Delta_{\lambda_i},x)=\prod_{i=1}^k [\l_i]_x$, as desired.

 \end{proof}
\begin{remark}
    Since $\F_{K_{n+1}}({\bf 1})$ is an element of $\F_{( (n-1,n-2, \ldots, 1), {\bf 1})}$, Theorem \ref{increasing-simplices} is a special case of Theorem \ref{thm:faces}.
\end{remark}
 
\bibliography{biblio-kir}
\bibliographystyle{plain}

\end{document}